\documentclass[10pt]{amsart}         

\usepackage{amssymb,amsmath, amsfonts}
\usepackage{hyperref}
\usepackage[usenames]{color}
\usepackage{latexsym}

\newtheorem{theorem}{Theorem}[section]

\newtheorem{lemma}[theorem]{Lemma}
\newtheorem{corollary}[theorem]{Corollary}
\newtheorem{proposition}[theorem]{Proposition}

\newtheorem{definition}[theorem]{Definition}

\numberwithin{equation}{section}

\newcount\quantno
\everydisplay{\quantno=0}\everycr{\quantno=0}
\newcommand\quant{\advance\quantno by1
                      \ifnum\quantno=1\qquad\else\quad\fi\forall }
\newcommand\itemno[1]{(\romannumeral #1)}

\newcommand\Dom{\operatorname{\text{Dom}}}

\newcommand\rest[1]{\kern-.1em
          \lower.5ex\hbox{$\scriptstyle #1$}\kern.05em}

\newcommand\set[1]{{\left\{#1\right\}}}

\renewcommand\mod[1]{\left\vert{#1}\right\vert}
\newcommand\bigmod[1]{\bigl\vert{#1}\bigr|}
\newcommand\Bigmod[1]{\Bigl\vert{#1}\Bigr|}

\newcommand\norm[2]{{\Vert{#1}\Vert_{#2}}}

\newcommand\opnorm[2]{|\!|\!| {#1} |\!|\!|_{#2}}

\newcommand\wrt{\,\text{\rm d}}

\newcommand\BC{\mathbb{C}}

\newcommand\BN{\mathbb{N}}
\newcommand\BR{\mathbb{R}}

\newcommand\cB{\mathcal{B}}

   \newcommand\frh{\mathfrak{h}}
\newcommand\cI{\mathcal{I}}

\newcommand\cL{\mathcal{L}}   
   \newcommand\frm{\mathfrak{m}}

\newcommand\cP{\mathcal{P}}

\newcommand\cT{\mathcal{T}}

  \newcommand\fX{\mathfrak{X}}

\newcommand\be{\beta}
\newcommand\ga{\gamma}    \newcommand\Ga{\Gamma}
\newcommand\de{\delta}

\newcommand\la{\lambda}   
      
\newcommand\si{\sigma}

\newcommand\vp{\varphi}

\newcommand\OV{\overline}
\newcommand\funnyk{k\hbox to 0pt{\hss\phantom{g}}}

\newcommand\lu[1]{L^1(#1)}
\newcommand\lp[1]{L^p(#1)}

\newcommand\ld[1]{L^2(#1)}

\newcommand\hu[1]{H^1(#1)}
\newcommand\huG[1]{\frh^1(#1)}

\newcommand\wt{\widetilde}
\newcommand\whH{\widehat{\phantom{G}}\hbox to 0pt{\hss $H$}}

\newcommand\emspace{\hbox to 6pt{\hss}}

\newcommand\rmi{\hbox{\rm (i)}}
\newcommand\rmii{\hbox{\rm (ii)}}
\newcommand\rmiii{\hbox{\rm (iii)}}

\newcommand\One{{\mathbf{1}}}
\newcommand\OU{{Ornstein--Uhlenbeck}\ }

\newcommand\e{\mathrm{e}}

\newcommand\EB{{\rm{(AMP)}}}
\newcommand\PP{{\rm{(I)}}}
\newcommand\PPc{{\rm{(I$^c$)}}}
\newcommand\supp{\mathrm{supp}\,}

\newcommand\ca{\textbf{1}}

\begin{document}

\title[]{Comparison of spaces of Hardy type \\
for the Ornstein--Uhlenbeck operator}


\thanks{Work partially supported by the
Progetto PRIN 2007 ``Analisi Armonica''.}

\author[A. Carbonaro, G. Mauceri and S. Meda]
{Andrea Carbonaro, Giancarlo Mauceri and Stefano Meda}

\address{A. Carbonaro, G. Mauceri: Dipartimento di Matematica \\
Universit\`a di Genova \\ via Dodecaneso 35, 16146 Genova \\ Italia}
\address{S. Meda: Dipartimento di Matematica e Applicazioni
\\ Universit\`a di Milano-Bicocca\\
via R.~Cozzi 53\\ 20125 Milano\\ Italy}

\begin{abstract}
Denote by $\ga$ the Gauss measure on $\BR^n$
and by $\cL$ the \OU operator.
In this paper we introduce a Hardy space $\huG{\ga}$ 
of Goldberg type and show that for each $u$ in $\BR\setminus\{0\}$
and $r>0$ 
the operator $(r\cI+\cL)^{iu}$ is unbounded from $\huG{\ga}$ to $\lu{\ga}$.  
This result is in sharp contrast
both with the fact that $(r\cI+\cL)^{iu}$ is bounded
from $\hu{\ga}$ to $\lu{\ga}$, where $\hu{\ga}$ denotes
the Hardy type space introduced in \cite{MM},
and with the fact that in the Euclidean case $(r\cI-\Delta)^{iu}$ is 
bounded from the Goldberg space $\huG{\BR^n}$ to $\lu{\BR^n}$.
We consider also the case of Riemannian manifolds $M$
with Riemannian measure $\mu$. 
We prove that, under certain geometric assumptions on $M$,
an operator $\cT$, bounded on $\ld{\mu}$,
and with a kernel satisfying certain analytic assumptions,
is bounded from $\hu{\mu}$ to $\lu{\mu}$
if and only if it is bounded from $\huG{\mu}$ to $\lu{\mu}$.
Here $\hu{\mu}$ denotes the Hardy space introduced in \cite{CMM1},
and $\huG{\mu}$ is defined in Section~\ref{s: Measured}, and is
equivalent to a space recently introduced by M.~Taylor \cite{T}.
The case of translation invariant operators on homogeneous trees is
also considered.
\end{abstract}

\maketitle

\section{Introduction} \label{s: Introduction}

Denote by $\ga$ the Gauss measure on $\BR^n$, i.e. the probability
measure with density $x\mapsto\pi^{-n/2}\, \e^{-\mod{x}^2}$ 
with respect to the Lebesgue measure. 

Harmonic analysis on the measured metric space $(\BR^n,d,\ga)$, where
$d$ denotes the Euclidean distance on $\BR^n$, has been the object
of many investigations.  In particular, efforts have been made
to study operators related to the \OU\ semigroup, with emphasis
on maximal operators \cite{S,GU,MPS1, GMMST2},
Riesz transforms 
\cite{Mu, Gun, M1, Pisier, Pe, Gut, GST, FGS, FoS, GMST1, PS, U, DV}
and functional calculus \cite{GMST2,GMMST1,MMS}.

In \cite{MM} the authors defined an atomic Hardy type space 
$\hu{\ga}$ associated to~$\ga$.  We briefly recall its 
definition.  An Euclidean ball $B$ is called \emph{admissible} if 
\begin{equation} \label{f: 1-admissible} 
r_B \leq \min\bigl(1,1/\mod{c_B}\bigr),
\end{equation}
where $r_B$ and $c_B$ denote the radius and the centre of $B$ respectively. 
An $\hu{\ga}$-\emph{atom} is either the constant function $1$ or 
a function $a$ in $\lu{\ga}$,
supported in an admissible ball $B$, such that
\begin{equation}  \label{f: prop atom}
\norm{a}{2} \leq \ga(B)^{-1/2}
\qquad\hbox{and}\qquad
\int_{\BR^n} a \, \wrt \ga =0,
\end{equation}
where $\norm{a}{2}$ denotes the norm of $a$ with respect to the Gauss measure.
The space $\hu{\ga}$ is then the vector space of all functions 
$f$ in $\lu{\ga}$ that admit a decomposition of the form 
$\sum_j \la_j \, a_j$, where the $a_j$'s are $\hu{\ga}$-atoms 
and the sequence of complex numbers $\{\la_j\}$ is summable. 
The norm of $f$ in $\hu{\ga}$ is defined as the infimum
of $\sum_j\mod{\la_j}$ over all representations of $f$ as above.

Note that $\hu{\ga}$ is defined much as the atomic space $H^1$
on spaces of homogeneous type in the sense
of R.R.~Coifman and G.~Weiss \cite{CW}, but with a difference. 
Namely, 
only the exceptional atom and atoms with ``small support'',
i.e., with support contained in admissible balls,
appear in the definition of $\hu{\ga}$.
This difference may appear irrelevant at first sight,
but it is, in fact, quite subtle and has important consequences.   
It is motivated by the fact that 
the Gauss measure of an open set $A$ in $\BR^n$ 
far away from the origin is concentrated in a ``thin'' shell
near the boundary of $A$.  More precisely, the following 
quantitative estimate holds \cite[Lemma~3.2~\rmii]{MM}:
there exist a ball $B_0$ centred at the origin and a constant $C$ such that
for each sufficiently small positive number $\kappa$ 
and each open set $A$ contained in~$B_0^c$
\begin{equation} \label{f: boundary gamma}
\ga\bigl(\{ x \in \BR^n: d(x, A^c) \leq \kappa/\mod{x} \}\bigr)
\geq C\, \kappa \, \ga(A):
\end{equation}
here $d$ denotes the Euclidean distance. 
Observe that the measured metric space $(\BR^n, d, \ga)$ is nondoubling. 

One of the results in \cite{MM} is that if an operator $\cT$
is bounded on $\ld{\ga}$ and has an integral kernel that satisfies
a local H\"ormander's type integral condition (see (\ref{f: hormander})
below), then $\cT$ is bounded from $\hu{\ga}$ to $\lu{\ga}$,
and, consequently on $\lp{\ga}$ for all $p$ in $(1,2)$.  
This result applies, for instance, to the imaginary powers of the \OU operator
(see Section~\ref{s: Notation} for the precise definition), and,
\emph{a fortiori}, to the operators $(\cI+\cL)^{iu}$, where $u$ is in $\BR$. 

In the Euclidean setting D.~Goldberg \cite{G} defined a ``local''
space of Hardy type $\huG{\BR^n}$.  It is defined much as the 
atomic Hardy space $\hu{\BR}$, but atoms are now either standard atoms
supported in small balls or square integrable functions
supported on large balls satisfying the usual size condition,
but without any cancellations.

Recently M.~Taylor~\cite{T} defined and studied a
local Hardy space of Goldberg type in the setting of
Riemannian manifolds with bounded geometry.  
Taylor's definition has a natural analogue in the Gauss setting.
In Section~\ref{s: Notation} we shall define a \emph{local} Hardy space of
Goldberg type $\huG{\ga}$ associated to the Gauss measure.
The $\huG{\ga}$-atoms are either $\hu{\ga}$-atoms, 
or functions $a$ supported in a ball $B$ 
with $r_B = \min(1,1/\mod{c_B})$, 
and satisfying the size condition in (\ref{f: prop atom}), 
but \emph{not} the cancellation condition.  We shall show that
$\hu{\ga}$ is properly contained in~$\huG{\ga}$.
Clearly if $\cT$ is a bounded linear operator from $\huG{\ga}$
to $\lu{\ga}$, then it is also bounded from $\hu{\ga}$
to $\lu{\ga}$.  The converse implication fails.  This is one of the main
result of this paper.

Specifically, we shall prove that if $\cT$ is bounded
from $\huG{\ga}$ to $\lu{\ga}$ and its kernel~$k_{\cT}$ 
satisfies a local H\"ormander type condition, then $k_{\cT}$ 
is ``uniformly integrable at infinity'', i.e., 
$$
\sup_{y\in\BR^n} \int_{(2B_y)^c} \mod{k(x,y)}\wrt\ga(x)
<\infty;
$$
here we denote by $B_y$ the ball
with centre $y$ and radius $\min(1,1/\mod{y})$.
As a consequence, in Section~\ref{s: Imaginary powers} 
we shall prove that if $u$ is in $\BR  \setminus\{0\}$
and $r$ is positive, 
then the operator $(r\cI+ \cL)^{iu}$,
which is bounded from $\hu{\ga}$ to $\lu{\ga}$ \cite[Thm~7.2]{MM},
is unbounded from $\huG{\ga}$ to $\lu{\ga}$.

The analysis on the Gauss space described above
may be put into a wider perspective.  Consider on $\BR^n$ 
the Riemannian distance~$\rho'$, whose length element is given by
\begin{equation} \label{f: rho primo}
\wrt s^2
= (1+\mod{x}^2)\, \, (\wrt x_1^2+\cdots +\wrt x_n^2).
\end{equation}
It is not hard to check \cite{CMM2} that balls
of radius at most $1$ with respect to $\rho'$ are
``equivalent'' to admissible balls, i.e., balls with respect 
to the Euclidean distance satisfying condition (\ref{f: 1-admissible}).
Condition (\ref{f: boundary gamma}) is then equivalent to
the following
$$
\ga\bigl(\{ x \in \BR^n: \rho'(x, A^c) \leq \kappa \}\bigr)
\geq C\, \kappa \, \ga(A).
$$
In the terminology of \cite{CMM2}
the measured metric space $(\BR^n, \rho', \ga)$ 
possesses the so called complementary 
isoperimetric property (see \cite[Section~8]{CMM2}).

An theory of Hardy type spaces
on a fairly large class of measured metric spaces $(M,\rho,\mu)$ 
has recently been developed in \cite{CMM1,CMM2}.
In these papers we assume that $\mu$ is a locally doubling measure, that
$(M,\rho,\mu)$ possesses an approximate midpoint property
(see Section~\ref{s: Measured} below), and 
either the isoperimetric or the complementary isoperimetric
property, according to whether $\mu(M)$ is infinite or not.
When the theory constructed in \cite{CMM2} is applied
to the space $(M,\rho',\ga)$, then the Hardy space $\hu{\mu}$
defined in \cite{CMM2} coincides with the space $\hu{\ga}$ 
defined above for the Gauss measure.  

Analogues of the local Hardy space of Goldberg type
may also be defined in this more general setting.
A natural question is whether there are singular
integral operators which are bounded from 
$\hu{\mu}$ to $\lu{\mu}$ but unbounded from
$\huG{\mu}$ to $\lu{\mu}$.   In Section~\ref{s: Measured}
we consider the cases where $M$ is either a homogeneous tree,
or a Riemannian manifold with spectral gap and
Ricci curvature bounded from below.

In the case of trees
we prove that an operator invariant with respect to the group
of isometries of the tree is bounded from $\hu{\mu}$ to $\lu{\mu}$
if and only if it is bounded from $\huG{\mu}$ to $\lu{\mu}$.

In the case of manifolds we prove that if $\cT$ is a 
bounded linear operator on $\ld{\mu}$ and has a kernel $k$
satisfying 
$$
\sup_{y \in M} \int_{B(y, 2)^c} \mod{\nabla_1 k(x,y)} \wrt\mu(y)
<\infty,
$$
where $\nabla_1$ denotes the gradient with respect to the first 
variable, then $\cT$ is bounded from $\hu{\mu}$
to $\lu{\mu}$ if and only if it is bounded from $\huG{\mu}$ to $\lu{\mu$}.

Furthermore, if $M$ is a \emph{unimodular} Lie group and we endow 
$M$ with a left invariant Riemannian metric, then 
a linear operator $\cT$ bounded on $\ld{\mu}$ 
and with kernel satisfying a \emph{local H\"ormander type integral
condition} (see (\ref{f: hormander I}) below), 
is bounded from $\hu{\mu}$
to $\lu{\mu}$ if and only if it is bounded from $\huG{\mu}$ to $\lu{\mu$}.

We will use the ``variable constant convention'', and denote by $C,$
possibly with sub- or superscripts, a constant that may vary from place to 
place and may depend on any factor quantified (implicitly or explicitly) 
before its occurrence, but not on factors quantified afterwards.

\section{Notation and background information}
\label{s: Notation}

The norm of a function $f$ in $\lp{\mu}$ will simply be
denoted by $\norm{f}{p}$.  
If $\cT$ is a bounded linear operator on $\lp{\mu}$, 
we shall write $\opnorm{\cT}{p}$ for its $\lp{\mu}$ operator norm.

We shall consider linear operators $\cT$ on various measure
spaces $(M,\mu)$.  When we do, we often associate to $\cT$
its kernel, which is defined as follows.

\begin{definition}
\emph{Suppose that $\cT$ is a bounded linear operator on 
$\ld{\mu}$ and that $k$ is a function on $M\times M$, 
locally integrable off the diagonal, such that
for every bounded function $f$ with compact support
$$
\cT f(x)
= \int_{M}k(x,y) \, f(y)\wrt\mu(y)
\quant x\in M\setminus\supp{f}.
$$
Then we say that $k$ is the \emph{kernel} of $\cT$
with respect to the measure $\mu$.}
\end{definition}

In this section and in Section~\ref{s: Imaginary powers} 
Lebesgue spaces will be with respect to the Gauss measure.
In Section~\ref{s: Measured} we shall also consider
Lebesgue spaces with respect to different measures $\mu$
on quite general measured metric spaces.  

Now we define the Hardy space of Goldberg type $\huG{\ga}$.

\begin{definition} 
\emph{A \emph{global} atom $a$ 
is a function in $\ld{\ga}$ with support contained 
in a ball $B$ of radius exactly equal to $\min (1,1/\mod{c_B})$ such that 
$$
\norm{a}{2}\leq \ga(B)^{-1/2}.
$$
A $\huG{\ga}$-atom is either a $\hu{\ga}$-atom 
(see the definition at the beginning of the Introduction)
or a global atom.}
\end{definition}
 
\begin{definition}
\emph{The Hardy space of Goldberg type $\huG{\ga}$ is the vector
space of all functions~$f$ which admit a decompositions of the form
\begin{equation} \label{f: Goldberg space}
f
=  \sum_j \la_j \, a_j,
\end{equation}
where the sequence $\{\la_j\}$ is summable 
and the $a_j$'s are $\huG{\ga}$-atoms.
The norm of $f$ in $\huG{\ga}$ is the infimum of $\sum_j \mod{\la_j}$
as $\{\la_j\}$ varies over all decompositions 
(\ref{f: Goldberg space}) of~$f$.}
\end{definition}

The $\hu{\ga}$-atoms and the global atoms that we consider are
often referred to as $(1,2)$-atoms.  In \cite{MM} it is shown that 
the space $\hu{\ga}$ may be defined in terms of the so-called
$(1,q)$-atoms, where $q$ is any number in $(1,\infty]$.  
A similar theory may also be developed for the space $\huG{\ga}$.
We omit the details. 

In the following proposition we shall make use of the space $BMO(\ga)$.
Recall that an integrable function is in $BMO(\ga)$ if 
$$
\norm{f}{BMO(\ga)} 
:= \norm{f}{1} + \sup_{B} \frac{1}{\ga(B)} \int_B \mod{f- f_B} \wrt \ga
< \infty,
$$
where the supremum is with respect to all admissible balls and
$$
f_B = 
\frac{1}{\ga(B)} \int_B f \wrt \ga.
$$ 
It is known that $BMO(\ga)$ is the Banach dual of $\hu{\ga}$
\cite[Thm~5.2]{MM}.

\begin{proposition}
The inclusion $\hu{\ga}\subset \huG{\ga}$ is strict.
\end{proposition}

\begin{proof}
For the sake of simplicity  we consider only the case
where $n=1$.  

First we show that the monomial $x\mapsto x^2$ is in $BMO(\ga)$.
Indeed, denote by $I$ any admissible interval, with centre $c_I$
and radius $r_I$.  Observe that 
$$
\begin{aligned}
\mod{x^2 - c_I^2} 
& =    \mod{x - c_I}\, \mod{x+c_I} \\
& \leq r_I \, (r_I+2\mod{c_I})
\quant x \in I.
\end{aligned}
$$
If $\mod{c_I} \geq 1$, the right hand side may be estimated
by $3r_I \mod{c_I}$, which is bounded by $3$ because $I$ is
admissible.  If $\mod{c_I} \leq 1$, then the right hand side
is at most $r_I \, (r_I+2)$, which is dominated by $3$, because $r_I \leq 1$.
Therefore
$$
\int_{I} \mod{x^2-c_I^2} \wrt \ga(x) 
\leq 3 \, \ga(I),
$$
so that $x\mapsto x^2$ is in $BMO(\ga)$. 

Suppose, by contradiction, that $\hu{\ga}=\huG{\ga}$. 
Then, by the closed graph theorem there exists a constant 
$C$ such that 
$\norm{f}{\hu{\ga}}
\le C \norm{f}{\huG{\ga}}$ for all functions $f$ in $\huG{\ga}$. 
In particular,
\begin{equation}\label{contra}
\norm{\One_{I}/\gamma(I)}{\hu{\ga}}\le C
\end{equation}
for all maximal admissible intervals $I$.

Since the integral $\int_I x^2\wrt\ga$ is absolutely convergent, 
the pairing between $\One_I$ and the function 
$x\mapsto x^2$ is given by $\int_I x^2\wrt\gamma$ 
(this follows from \cite[Thm~5.2]{MM} and the fact that 
$BMO(\gamma)$ is a lattice, as in \cite[IV.1.2]{St2}).

Now observe that, if $\mod{c_I}$ is sufficiently large,
$$
\begin{aligned}
\norm{x^2}{BMO(\ga)} \,\norm{\One_I}{\hu{\ga}}
& \geq \, \int_I \, x^2 \wrt\ga(x)  \\
& \ge \, \frac{\mod{c_I}^2}{2} \, \ga(I),
\end{aligned}
$$
so that the supremum of the $\hu{\ga}$-norms of the functions $\One_I/\ga(I)$
is unbounded as $I$ varies over all maximal admissible intervals, contradicting (\ref{contra}).
\end{proof}

\section{Imaginary powers of the \OU operator}
\label{s: Imaginary powers}

The \OU operator $\cL$ is
the closure in $\ld{\ga}$ of the operator $\cL_0$,
defined by
$$
\cL_0
= -\frac{1}{2}\,  \Delta + x\cdot\nabla
\quant f \in C^\infty_c(\BR^n),
$$
where $\Delta$ and $\nabla$ denote the Euclidean
Laplacian and gradient respectively.
The spectral resolution of the identity of $\cL$ is
$$
\cL f
= \sum^{\infty}_{j=0} j\, \cP_jf  
\quant f\in\Dom(\cL),
$$
where $\cP_j$ is the orthogonal projection onto the linear span 
of Hermite polynomials of degree~$j$ in $n$ variables.  
For each $u$ in $\BR$
consider the sequence $M_{u}:\BN\rightarrow \BC$, defined by
$$
M_{u}(j)=
\begin{cases}
0       & {\rm if}\ j=0\\
j^{iu}  & {\rm if}\ j=1,2,\ldots
\end{cases}
$$
The family of (spectrally defined) operators $\set{M_u(\cL)}_{u\in\BR}$ 
will be referred to as {\it imaginary powers} of the \OU operator. 
They are bounded on $\lp{\ga}$ for every 
$p$ in $(1,\infty)$, by the general 
Littlewood--Paley--Stein theory for generators of symmetric 
diffusion semigroups \cite{St1}. Sharp estimates 
of the behavior of their norms on $\lp{\ga}$ 
as $\mod{u}$ tends to infinity have been given in 
\cite{GMMST1} and \cite{MMS}, where the estimates are used 
to prove spectral multiplier theorems. 
It is also known that they are of weak type $(1,1)$ \cite{GMST2} 
and bounded from $\hu{\ga}$ to $\lu{\ga}$ \cite{MM}. 

In this section we shall show that for each
$u$ in $\BR\setminus \{0\}$ the operator $(\cI + \cL)^{iu}$
is unbounded from $\huG{\ga}$ to $\lu{\ga}$.
Slight modifications of the proof show that a similar result holds
for all $r$ in $\BR^+$ with $(r\cI + \cL)^{iu}$ in place of $(\cI + \cL)^{iu}$.

This result is in sharp contrast with the Euclidean case.  
Indeed,  it is well known \cite{G} that 
the operator $(\cI-\Delta)^{iu}$ is 
bounded from $\huG{\BR^n}$ to $\lu{\BR^n}$.

We shall need the following lemma.

\begin{lemma} \label{l: MSV}
Suppose that $\cT$ is a $\lu{\ga}$-valued linear operator 
defined on finite linear combinations of $\huG{\ga}$-atoms.  
The following are equivalent:
\begin{enumerate}
\item[\itemno1]
$\cT$ extends to a bounded operator from $\huG{\ga}$ to $\lu{\ga}$;
\item[\itemno2]
$\sup \{ \norm{\cT a}{1}: \hbox{\emph{$a$ is a $\huG{\ga}$-atom}} \}
<\infty$.
\end{enumerate}
\end{lemma}

\begin{proof}
Clearly \rmi\ implies \rmii, for every $\huG{\ga}$-atom has
$\huG{\ga}$ norm at most $1$.

The converse in nontrivial.   However, it is not hard to adapt
the proof of \cite[Thm~4.1]{MSV} to the present case.
We omit the details. 
\end{proof}

\begin{definition}
\emph{Suppose that $\cT$ is an operator with kernel $k$. 
We say that $k$ satisfies a \emph{local integral condition
of H\"ormander type} if 
\begin{equation}\label{f: hormander}
H_k
: = \sup_{B}
\sup_{y,y'\in B} \int_{(2B)^c} \mod{k(x,y)-k(x,y')}\wrt\ga(x)
<\infty,
\end{equation}
where $2B$ denotes the ball with the same centre as $B$ and
twice the radius and the supremum is taken with 
respect to all admissible balls.}
\end{definition}

It is known \cite[Thm~7.2]{MM} that
if the kernel $k$ satisfies the local H\"ormander condition above, 
then $\cT$ extends to a bounded operator from $\hu{\ga}$ to $\lu{\ga}$. 
We aim at showing that $\cT$ may be unbounded from $\huG{\ga}$ to $\lu{\ga}$. 
We shall use the following simple criterion.

\begin{proposition}\label{p: kernelL1}
Suppose that $\cT$ is a bounded linear operator on $\ld{\ga}$
with kernel~$k$.
The following hold:
\begin{enumerate}
\item[\itemno1]
if $\cT$ is bounded from $\huG{\ga}$ to $\lu{\ga}$ and $k$ satisfies
the local H\"ormander type condition (\ref{f: hormander}), then 
$k$ satisfies the following estimate
\begin{equation} \label{f: UI at infinity}
I_\infty
: = \sup_{y\in\BR^n} \int_{(2B_y)^c} \mod{k(x,y)}\wrt\ga(x)
<\infty,
\end{equation}
where, for every $y$ in $\BR^n$ we denote by $B_y$ the ball
with centre $y$ and radius $\min(1,1/\mod{y})$;
\item[\itemno2]
if $\cT$ is bounded from $\hu{\ga}$ to $\lu{\ga}$ and
$k$ satisfies (\ref{f: UI at infinity}), then $\cT$ is bounded
from $\huG{\ga}$ to $\lu{\ga}$.
\end{enumerate}
\end{proposition}

\begin{proof}
First we prove \rmi.
Since $\cT$ is bounded from $\huG{\ga}$ to $\lu{\ga}$,
the following holds
\begin{equation}
A
:= \sup \{ \norm{\cT a}{1}: \hbox{$a$ is a global atom} \}
< \infty,
\end{equation}
because $\huG{\ga}$-atoms have $\huG{\ga}$-norm at most $1$.
The function $a_{y} = \One_{B_y}/\ga(B_y)$ is a global atom at the scale $1$. 
Notice that
\begin{equation} \label{f: Tay}
\begin{aligned}
\cT a_y (x)
& = \int_{B_y} k(x,v) \, a_y(v) \wrt \ga(v)  \\
& = \int_{B_y} \bigl[k(x,v)-k(x,y)\bigr] \, a_y(v) \wrt \ga(v) 
     + k(x,y) \, \int_{B_y} a_y(v) \wrt \ga(v)  \\
& = \frac{1}{\ga(B_y)} \int_{B_y} \bigl[k(x,v)-k(x,y)\bigr] \wrt \ga(v) 
     + k(x,y).
\end{aligned}
\end{equation}
Thus, 
$$
\begin{aligned}
\int_{(2B_y)^c} \mod{k(x,y)}\wrt\ga(x)
&\leq \sup_{v \in B_y}\int_{(2B_y)^c} \bigmod{k(x,v)-k(x,y)}\wrt \ga(x) 
     +  \norm{\cT a_y}{1}  \\ 
&\leq H_k +  \norm{\cT a_y}{1}.
\end{aligned}
$$
By taking the supremum over $y$ in $\BR^n$, we obtain
$$
\sup_{y\in \BR^n}
\int_{(2B_y)^c} \mod{k(x,y)}\wrt\ga(x)
\leq H_k +  A,
$$
as required.

Now we prove \rmii. 
Since $\cT$ is bounded from $\hu{\ga}$ to $\lu{\ga}$,
$\cT$ is uniformly bounded on $\hu{\ga}$-atoms.  In view
of Lemma~\ref{l: MSV} to prove that $\cT$ is bounded from 
$\huG{\ga}$ to $\lu{\ga}$ it suffices to show that 
$\cT$ is uniformly bounded on global atoms at the scale $1$.

Suppose that $a$ is a global atom at the scale $1$, 
with support contained in $B_y$. 
Clearly
\begin{equation}\label{f: 1}
\norm{\cT a}{1}
= \norm{\ca_{2B_y}\cT a}{1}+\norm{\ca_{(2B_y)^c}\cT a}{1}.
\end{equation}
It is not hard to check that there exists a constant $C$,
independent of $y$ in $\BR^n$, such that
$\ga (2B_y)^{1/2} \leq C \, \ga (B_y)^{1/2}$ (see \cite[prop.~2.1~\rmii]{MM}).
Therefore
$$
\begin{aligned}
\norm{\ca_{2B_y} \, \cT a}{1}
& \leq   \ga (2B_y)^{1/2} \, \,  \norm{\cT a}{2} \\
&\leq C \, \ga (B_y)^{1/2}\,  \, \opnorm{\cT}{2} \, \norm{a}{2}\\
&\leq C \, \opnorm{\cT }{2}.
\end{aligned}
$$
Furthermore
\begin{equation} \label{f: 3} 
\begin{aligned}
\norm{\ca_{(2B_y)^c}\cT a}{1}
&\leq \int_{\BR^n} \wrt\ga(y)\, \mod{a(y)}\int_{(2B_y)^c}
      \mod{k(x,y)}\wrt\ga(x)\\
&\leq I_\infty \, \norm{a}{1}\nonumber\\
&\leq I_\infty.
\end{aligned}
\end{equation}
Hence 
$$
\norm{\cT a}{1} 
\leq  C \, \opnorm{\cT }{2} + I_\infty,
$$
with $C$ independent of $a$, 
as required to conclude the proof of the proposition.
\end{proof}

Fix $u$ in $\BR \setminus \{0\}$ and $r$ in $\BR^+$. 
The kernel $k$ of the operator $(\cL+r\cI)^{iu}$ (with respect to the
Gauss measure) is given~by
\begin{equation} \label{f: nucleo}
k(x,y)
= \frac{1}{\Ga(iu)} \int^\infty_0 t^{-iu-1} \, \e^{-rt} \, h_t(x,y)\wrt t
\quant x,y\in\BR^n,\ x\neq y,
\end{equation}
where $\Ga$ denotes the Euler function and $h_t$ is the Mehler's kernel,
i.e. the kernel of the operator $\exp(-t\cL)$ \cite{GMMST1}
with respect to the Gauss measure.  Recall the formula
\begin{equation} \label{f: Mehler}
h_t(x,y)
= \frac{1}{(1-\e^{-2t})^{d/2}} \, 
\, \exp\Bigl[\mod{y}^2- \frac{\mod{\e^{-t}\, x - y}^2}{1-\e^{-2t}} \Bigr]
\end{equation}
where $t$ is in $\BR^+$ and $x$ and $y$  are in $\BR^n$. 
We perform the change of variables $t=\log((1+s)/(1-s))$
in (\ref{f: nucleo}).  This change of variables, which
was first introduced in \cite{GMMST1}, transforms the Mehler kernel to
\begin{equation} \label{f: Mehler modif}
\wt h_s(x,y)
= \frac{(1+s)^n}{(4s)^{d/2}} \, 
\, \exp\Bigl[\frac{\mod{x}^2+\mod{y}^2}{2} 
- \frac{1}{4} \, (s\mod{x + y}^2 + s^{-1} \, \mod{x-y}^2)\Bigr],
\end{equation}
and the kernel $k$ is expressed via the following formula
\begin{equation} \label{f: nucleo modif}
k(x,y)
= \frac{1}{\Ga(iu)}
 \int^\infty_0 \frac{g_u(s)}{1+s} \,  
\e^{-Q_s(x,y)}\frac{\wrt s}{s^{1/2}}
\quant x,y\in\BR^n,\ x\neq y.
\end{equation}
where $Q_s$ denotes the quadratic form
$$
Q_s(x,y)
= \frac{1}{2} \, (\mod{x}^2+\mod{y}^2)
- \frac{1}{4} \, \Bigl(\frac{\mod{x-y}^2}{s}+s\mod{x+y}^2 \Bigr),
$$
and $g_u: (0,1) \to \BC$ is the function defined by
$$
g_u(s)
= \Bigl[\log\Bigl(\frac{1+s}{1-s}\Bigr)\Bigr]^{-iu-1}.
$$
Following \cite{GMMST1}, for every $a$ in $\BR^+$  define the function $F_a$ 
$$
F_a(s)
= -a(s-1)^2/4s
\quant s \in (0,1)
$$ 
and
$$
I(a,\si)
= \int_0^1 g_u(s) \, \frac{\e^{F_a(s/\si)}}{1+s} \frac{\wrt s}{s^{1/2}}.
$$
It is straightforward, though tedious,
to check that in the case where $n=1$ the following formula holds
\begin{equation} \label{f: kernel Gauss}
k(x,y)
= \e^{y^2} \, I\bigl(\mod{x^2-y^2}, \mod{x-y}/\mod{x+y}\bigr).
\end{equation}
The following lemma, which is reminiscent of \cite[Lemma~4.2]{GMMST1},
will be key to obtain precise estimates of $k$.

\begin{lemma}\label{l: 1}
There exists a positive constant $C$ such that
$$
\mod{I(a,\si)}
\geq \frac{C}{\sqrt{a\, \si}}
\quant a \in [1,\infty) \quant \si \in (0,1/2]. 
$$
\end{lemma}

\begin{proof}
It will be convenient to define two more functions, $J$ and $H$,
by the formulae
\begin{equation} \label{f: defin J e H}
J(a,\si)
= g_u(\si) \, \int^{2/3}_{\si/2}\frac{\e^{F_a(s/\si)}}{1+s} \, 
\frac{\wrt s}{s^{1/2}}
\qquad\hbox{and} \qquad
H(a,\si) = I(a,\si) - J(a,\si).
\end{equation}
We claim that there exist $C, M>0$ such that
$$
\mod{J(a,\si)}
\geq \frac{C}{\sqrt{a\si}} 
\qquad\hbox{and}\qquad
\mod{H(a,\si)}\leq \frac{M}{a\, \sqrt{\si}}.
$$
The required estimate on $I$ will follow directly from
the claim.  

To prove the claim, define $H^1$, $H^2$ and $H^3$ by
$$
H^1(a,\si)
= \int_0^{\si/2} g_u(s)\, \frac{\e^{F_a(s/\si)}}{1+s} 
\, \frac{\wrt s}{s^{1/2}}
\qquad\qquad
H^2(a,\si)
= \int^1_{2/3} g_u(s)\, \frac{\e^{F_a(s/\si)}}{1+s} \, \frac{\wrt s}{s^{1/2}}
$$
and 
$$
\begin{aligned}
H^3(a,\si)
= \int^{2/3}_{\si/2}\left(g_u(s)-g_u(\si)\right)
\, \frac{\e^{F_a(s/\si)}}{1+s} \, \frac{\wrt s}{s^{1/2}}.
\end{aligned}
$$
Clearly $H = H^1+H^2+H^3$. 
Since $\si\leq 1/2$, there exist positive constants $C$ and $c$ such that
\begin{equation}\label{f: H1}
\begin{aligned}
\mod{H^1(a,\si)}
&\leq \int^{\si/2}_0
     \Bigl[\log\Bigl(\frac{1+s}{1-s}\Bigr)\Bigr]^{-1} \frac{s^{-1/2}}{1+s} \, 
     \e^{F_a(s/\si)}\wrt s\\
&\leq C\,\int^{\si/2}_0 s^{-3/2} \, \e^{F_a(s/\si)}\wrt s \\
&\leq C\,\int^{\si/2}_0 s^{-3/2} \, \e^{-ca\si/s}\wrt s\\
&\leq C\,(a\si)^{-1/2}\int^\infty_{2ca}s^{-1/2} \, \e^{-s}\wrt s\\
&\leq C\, a^{-1}\, \si^{-1/2} \, \e^{-ca}.
\end{aligned}
\end{equation}
A similar computation shows that for every 
$a\geq 1$ and $\si\leq 1/2$ there exists $C>0$ such that
\begin{equation}\label{f: H2}
\mod{H^2(a,\si)}\leq C\, a^{-1}\,\si^{-1/2}.
\end{equation}
Now we estimate $H^3(a,\si)$. 
Note that there exists $C>0$ such that
\begin{equation}\label{f: gder}
\Bigmod{\frac{\wrt}{\wrt s}g_u(s)}
\leq C_u \, \Bigl[\log\Bigl(\frac{1+\si/2}{1-\si/2}\Bigr)\Bigr]^{-2}
\quant s\in(\si/2,2/3).
\end{equation}
Hence, by the mean value theorem, we have that
\begin{equation}\label{f: H3}
\begin{aligned}
\mod{H^3(a,\si)}
&\leq C_u\, \int^{2/3}_{\si/2} \mod{s-\si} \, 
    \Bigl[\log\Bigl(\frac{1+\si/2}{1-\si/2}\Bigr)\Bigr]^{-2}
    \frac{s^{-1/2}}{1+s} \, \e^{F_a(s/\si)}\wrt s\\
&\leq C\, \si^{-2}\int^{2/3}_{\si/2} \mod{s-\si}\, s^{-1/2}\,  \e^{F_a(s/\si)}
   \wrt s\\
&= C\,  \si^{-1/2}\int^{2/3\si}_{1/2}\mod{s-1}\,  \e^{F_a(s)}\wrt s\\
&\leq C \, \si^{-1/2}\int^{+\infty}_{1/2}\mod{s-1}\,  \e^{F_a(s)}\wrt s\\
&\leq C \, a^{-1}\si^{-1/2}.
\end{aligned}
\end{equation}
By combining (\ref{f: H1}), 
(\ref{f: H2}) and (\ref{f: H3}) we get the desired estimate for $H(a,\si)$.

Now we estimate $J(a,\si)$.  Observe that
$$
\begin{aligned}
\mod{J(a,\si)}
& =   \mod{g_u(\si)}\int^{2/3}_{\si/2}\frac{s^{-1/2}}{1+s}
     \, \e^{F_a(s/\si)}\wrt s\\
& \geq C \, \mod{g_u(\si)}\int^{2/3}_{\si/2} \e^{F_a(s/\si)}\wrt s\\
& \geq C \, \Bigl[\log\Bigl(\frac{1+\si}{1-\si}\Bigr)\Bigr]^{-1} \,
    \sqrt{\frac{\si}{a}}\\
& \geq C \, (a\, \si)^{-1/2},
\end{aligned}
$$
as required.
\end{proof}

\begin{theorem}\label{t: 1}
For each $u$ in $\BR\setminus\{ 0\}$ and for each $r$ in $\BR^+$
the operator $(r\cI+\cL)^{iu}$ is unbounded from $\huG{\ga}$ to $\lu{\ga}$.
\end{theorem}

\begin{proof}
We prove the result when $r=1$.  The modifications needed
to prove the result for $r>0$ are straighforward
and omitted.

A slight modification of the proof of \cite[Thm~7.2]{MM} shows that 
the kernel $k$ of $(\cI+\cL)^{iu}$ 
satisfies H\"ormander's type condition (\ref{f: hormander}). 
Thus, by Proposition~\ref{p: kernelL1}, 
to prove the theorem it suffices to show that
\begin{equation}\label{f: kernelatinfity}
\lim_{\mod{y}\to\infty}  \int_{(2B_y)^c} \mod{k(x,y)} \wrt\ga(x)
=  \infty,
\end{equation}
where $B_y$ denotes the ball with centre $y$ and radius $\min(1,1/\mod{y})$.

We shall give the details only in the case where $n=1$. 
The proof in the case where $n\geq 2$ is more technical,
but it follows the same lines.  See also the proof of 
\cite[Proposition~4.4]{GMMST1}, where similar computations
are made in all dimensions and the differences between
the one dimensional and the higher dimensional cases are explained in detail.  

By (\ref{f: kernel Gauss}), it suffices to prove that the function
$$
y\mapsto 
\int_{(2B_y)^c} \bigmod{I\bigl(\mod{x^2-y^2},\mod{x-y}/\mod{x+y}\bigr)} 
\wrt x 
$$
is unbounded.  We may restrict the domain of integration
to the set where $y$ is large and positive, and
$x$ is in the interval $(y-1,y-2/y)$.
Then we must prove that
\begin{equation} \label{f: last comp}
\lim_{y \to \infty}
\int_{y-1}^{y-2/y} \bigmod{I\bigl({y^2-x^2},(y-x)/(x+y)\bigr)}  
\wrt x  = \infty.
\end{equation}
Note that in the interval $(y-1,y-2/y)$
$$
\bigmod{I\bigl({y^2-x^2},(y-x)/(x+y)\bigr)} 
\geq C \, (y-x)^{-1}.
$$
Indeed, in that interval $y^2-x^2 \geq (2/y) \, (x+y)\geq 2$,
and $(y-x)/(x+y)\leq 1/2$, so that Lemma~\ref{l: 1} may be applied,
and the estimate above follows.
Therefore the limit in (\ref{f: last comp}) is estimated
from below by 
$$
\lim_{y\to \infty}
\int_{y-1}^{y-2/y} (y-x)^{-1}  \wrt x  = \infty,
$$
as required to conclude the proof of the theorem.
Now (\ref{f: last comp}) follows directly from this estimate.
\end{proof}

\section{Measured metric spaces} \label{s: Measured}

We recall briefly the relevant definition and
refer to \cite{CMM1,CMM2} and the references therein
for every unexplained notation and terminology and for more on
measured metric spaces.   

Suppose that $(M,\rho,\mu)$ is a measured metric space. 
In particular, we assume that $(M,\rho)$ is a metric space, that $\mu$
is a regular Borel measure on $M$
with the property that $\mu(M) > 0$ and every ball has finite measure.
We assume throughout that $M$ is \emph{unbounded}. 
We denote by $\cB$ the family of all balls on $M$.
For each $B$ in $\cB$ we denote by $c_B$ and $r_B$
the centre and the radius of $B$ respectively,
and by $\kappa \, B$ the
ball with centre $c_B$ and radius $\kappa \, r_B$.
For each $b$ in $\BR^+$, we denote by $\cB_b$ the family of all
balls $B$ in $\cB$ such that $r_B \leq b$.  
For any subset $A$ of $M$ and each $\kappa$ in $\BR^+$
we denote by $A_{\kappa}$ and $A^{\kappa}$ the sets
$$
\bigl\{x\in A: \rho(x,A^c) \leq \kappa\bigr\}
\qquad\hbox{and}\qquad
\bigl\{x\in A: \rho(x,A^c) > \kappa\bigr\}
$$
respectively.

We say that the measured metric space $(M,\rho,\mu)$
possesses the \emph{local doubling property} (LDP) if
for every $b$ in $\BR^+$ there exists a constant $D_b$ 
such that
\begin{equation}  \label{f: LDC} 
\mu \bigl(2 B\bigr)
\leq D_b \, \mu  \bigl(B\bigr)
\quant B \in \cB_b. 
\end{equation}

We say that the measured metric space $(M,\rho,\mu)$
with $\mu(M) = \infty$
possesses the \emph{isoperimetric property} \hbox{\PP} if
there exist $\kappa_0$ and $C$ in $\BR^+$ such that
for every bounded open set $A$
\begin{equation}
\label{f: IP}
\mu \bigl(A_{\kappa}\bigr)
\geq C \, \kappa\, \mu (A) \quant \kappa\in (0,\kappa_0]. 
\end{equation}

It is known \cite{CMM1} that if $M$ is a complete Riemannian manifold,  the 
isoperimetric property (defined in terms of the Riemannian distance
and the Riemannian volume) is equivalent to the positivity of 
Cheeger's isoperimetric costant $h(M)$, defined by 
$$
h(M)
= \inf  \, \frac{\sigma \bigl(\partial(A)\bigr)}{\mu(A)},
$$
where the infimum runs over all bounded open sets $A$ with smooth
boundary.
Here $\sigma$ denotes the induced Riemannian measure on $\partial A$. 
Moreover, if the Ricci curvature of $M$ is bounded from below, both properties are 
equivalent to the existence of a spectral gap for the Laplacian.

The analogue of the isoperimetric property for measured
metric spaces of finite measure is the so-called
complementary isoperimetric inequality, which we now define.
We say that a measured metric space $(M,\rho,\mu)$ of finite measure
possesses the \emph{complementary isoperimetric property} \hbox{\PPc} if
there exist a ball $B_0$ in $M$, 
$\kappa_0$ and $C$ in $\BR^+$ such that
for every bounded open set $A$ contained in $M\setminus \Bar B_0$
\begin{equation}\label{PIc}
\mu \bigl(A_{\kappa}\bigr)
\geq C\, \kappa \, \mu (A) 
\quant \kappa \in (0,\kappa_0]. 
\end{equation}

We say that the measured metric space $(M,\rho,\mu)$
possesses the \emph{property} \hbox{\EB} (approximate midpoint property)
if there exist $R_0$ in $[0,\infty)$ and 
$\be$ in $(1/2,1)$ such that for every pair of points~$x$ and $y$
in $M$ with $\rho(x,y) > R_0$ there exists a point $z$ in $M$ such that
$\rho(x,z) < \beta\, \rho (x,y)$ and $\rho(y,z) < \beta\, \rho (x,y)$.

Clearly every length measured metric space possesses property (AMP).
The measured metric space $(\BR^n, \rho', \ga)$ ($\rho'$ is 
as in (\ref{f: rho primo})) is a locally doubling measured metric space
with the complementary isoperimetric and the approximate
midpoint property.

We briefly recall the definition of the Hardy space $\hu{\mu}$
in this setting \cite{CMM1,CMM2}. 

\begin{definition} \label{d: standard atom}
\emph{A (standard) \emph{atom} 
$a$ is a function in $\ld{\mu}$ supported in a ball $B$ in $\cB$
such that
$$
\norm{a}{2}  \leq \mu (B)^{-1/2}
\qquad\hbox{and}\qquad
\int_B a \wrt \mu  = 0.
$$}
\end{definition}

\begin{definition} \label{d: Hardy}
\emph{Suppose that $\mu(M) = \infty$.
The \emph{Hardy space} $\hu{\mu}$ is the 
space of all functions~$g$ in $\lu{\mu}$
that admit a decomposition of the form
$$
g = \sum_{k=1}^\infty \la_k \, a_k,
$$
where $a_k$ is an atom \emph{supported in a ball $B$ of radius 
at most $1$},
and $\sum_{k=1}^\infty \mod{\la_k} < \infty$.
The norm $\norm{g}{\hu{\mu}}$
of $g$ is the infimum of $\sum_{k=1}^\infty \mod{\la_k}$
over all decompositions of $g$ as above.}
\end{definition}

In the case where $\mu$ is finite in addition to the standard
atoms defined above there is also an
\emph{exceptional atom}, i.e. the constant function
$1/\mu(M)$.  The \emph{Hardy space} $\hu{\mu}$ is 
defined as in the case where $\mu(M) = \infty$,
but now atoms are either standard atoms or the
exceptional atom.
These atoms will be referred to as $\hu{\mu}$-atoms.

To avoid technicalities we assume throughout that $R_0/(1-\be) < 1$.  
In view of \cite[Prop.~4.3]{CMM1} and \cite[Prop.~3.4~\rmi]{CMM2} this ensures that
the Hardy space $\hu{\mu}$ defined above is scale invariant
in the following sense.  For $b> (R_0/(1-\be)$
we may consider an Hardy space $H_b^1(\mu)$
defined as in Definition~\ref{d: Hardy}, but where atoms are
supported in balls of radius at most $b$ instead that $1$.  With this notation
the space $\hu{\mu}$ defined above would be denoted by $H_1^1(\mu)$.
It is a nontrivial fact that 
the spaces $H_1^1(\mu)$ and $H_b^1(\mu)$ agree as vector spaces
and that their norms are equivalent. 

Now we define the Goldberg type space $\huG{\mu}$ in this setting.

\begin{definition} 
\emph{A \emph{global} atom $a$ (at the scale $1$) 
is a function in $\ld{\mu}$ with support contained 
in a ball $B$ of radius exactly equal to $1$ such that 
$$
\norm{a}{2}\leq \mu(B)^{-1/2}.
$$
An $\huG{\mu}$-atom is either an $\hu{\mu}$-atom or a global atom.}
\end{definition}

\begin{definition}
\emph{The Hardy space of Goldberg type $\huG{\mu}$ is the vector
space of all functions~$f$ which admits a decompositions of the form
\begin{equation} \label{f: Goldberg space I}
f
=  \sum_j \la_j \, a_j,
\end{equation}
where the sequence $\{\la_j\}$ is summable 
and the $a_j$'s are $\huG{\mu}$-atoms.
The norm of~$f$ in $\huG{\mu}$ is the infimum of $\sum_j \mod{\la_j}$
as $\{\la_j\}$ varies over all decompositions 
(\ref{f: Goldberg space I}) of $f$.}
\end{definition}

If $\mu$ is infinite, then $\hu{\mu}$ is contained in 
the space of integrable functions with integral $0$.
Since global atoms in $\huG{\mu}$ are integrable functions
with possibly nonzero integral, the strict inclusion
$\hu{\mu} \subset \huG{\mu}$ holds also in this case.

An equivalent space has been defined and studied on Riemannian
manifolds with  bounded geometry by M.~Taylor~\cite{T}.
In fact, the definition of Taylor is different
from that adopted above (see \cite[Section~2]{T}),
but it is straightforward to check
that the two definitions are equivalent, i.e., the corresponding
spaces agree, with equivalent norms.   

Assume that $\cT$ is a bounded linear operator on 
$\ld{\mu}$ with kernel $k$ (see Section~\ref{s: Notation} for the definition).
In \cite[Thm~8.2]{CMM1} it has been proved that, 
if $k$ satisfies the following \emph{local H\"ormander type condition}
\begin{equation}\label{f: hormander I}
H_k
= \sup_{B} \, \sup_{y,y'\in B}
\int_{(2B)^c} \mod{k(x,y)-k(x,y')}\wrt\mu (x)
<\infty,
\end{equation}
where the supremum is taken over all balls $B$ of radius at most $1$,
then $\cT$ extends to a bounded operator from $\hu{\mu}$ to $\lu{\mu}$. 
It is natural to speculate under what conditions the operator 
$\cT$ extends to a bounded operator from $\huG{\mu}$ to $\lu{\mu}$. 
The following is the analogue of Proposition~\ref{p: kernelL1} above.

\begin{proposition}\label{p: kernelL1 mms}
Suppose that $\cT$ is a bounded linear operator on $\ld{\mu}$
with kernel~$k$.
The following hold:
\begin{enumerate}
\item[\itemno1]
if $\cT$ is bounded from $\huG{\mu}$ to $\lu{\mu}$ and $k$ satisfies
the local H\"ormander type condition (\ref{f: hormander I}), then 
$k$ satisfies the following estimate
\begin{equation}\label{f: kernelinL1 I}
I_\infty
:= \sup_{y\in M} \int_{B(y,2)^{c}} \mod{k(x,y)}\wrt\mu(x)
<\infty;
\end{equation}
\item[\itemno2]
if $\cT$ is bounded from $\hu{\mu}$ to $\lu{\mu}$ and
$k$ satisfies (\ref{f: kernelinL1 I}), then $\cT$ is bounded
from $\huG{\mu}$ to $\lu{\mu}$.
\end{enumerate}
\end{proposition}

\begin{proof}
The proof is, \emph{mutatis mutandis}, the same as the proof
of Proposition~\ref{p: kernelL1}.  We only need to replace
the ball $B_y$ in that proof with the ball $B(y,1)$.   We omit
the details.
\end{proof}

\subsection{Homogeneous trees}

We now show that there are cases in which boundedness 
from $\hu{\mu}$ to $\lu{\mu}$ is equivalent
to boundedness from $\huG{\mu}$ to $\lu{\mu}$.
This is in sharp contrast with the 
case of the Gauss measure which has been
analysed in the Section~\ref{s: Imaginary powers}. 

Denote by $\fX$ a homogeneous tree, i.e., a graph, with no
loops, in which every vertex $x$ has the same number,
$q+1$ say, of adjacent vertices, called nearest neighbours of $x$.  
When $x$ and $y$ are adjacent vertices, we shall write $x\sim y$.
Denote by $\mu$ the \emph{counting measure} on $\fX$,
and by $\rho$ one-half of the natural distance on $\fX$.  
Thus, two adjacent vertices have distance $1/2$.  
The reason for this apparently unnatural definition of distance
is that if the distance of two adjacent vertices were equal to $1$,
then the only atom supported on any ball of radius at most $1$
would be the trivial atom.   
We could, of course, consider
balls of radius at most $2$, but then this would 
require new definitions and there would not be uniformity
with the Gaussian case and the case of manifolds.  

Denote by $G$ the group of isometries of $\fX$
(see \cite{FTN} for information on $G$) and fix a reference 
point $o$ in $\fX$.  We shall consider
only $G$-invariant linear operators acting
on function spaces on $\fX$.  If $\cT$ is such an operator,
then its kernel $k$ satisfies the following
$$
k(x,y)
= k(g\cdot x, g \cdot y)
\quant g \in G \quant x,y \in \fX,
$$
so that $k(x,y)$ depends, in fact, only on $\rho(x,y)$.
As a consequence, the local H\"ormander type condition
(\ref{f: hormander I}) may be reformulated thus
\begin{equation} \label{f: LHIC reform}
\max_{y\sim o}
\sum_{x: \rho(x,o) \geq 2} \mod{k(x,y)-k(x,o)}
< \infty.
\end{equation}

\begin{proposition}
Suppose that $\cT$ is a $G$-invariant linear operator defined on 
functions on $\fX$ with finite support and denote by $k$ its kernel.
The following hold:
\begin{enumerate}
\item[\itemno1]
if $\cT$ extends to a bounded operator from $\hu{\fX}$
to $\lu{\fX}$, then $k$ satisfies a
H\"ormander type integral condition;
\item[\itemno2]
if $k$ satisfies a local H\"ormander type integral condition, then 
$$
\sum_{x \in \fX} \mod{k(x,o)}
< \infty.
$$
Hence $\cT$ is bounded on $\lu{\fX}$;
\item[\itemno3]
$\cT$ extends to a bounded operator from $\hu{\fX}$
to $\lu{\fX}$ if and only if $\cT$ extends
to a bounded operator from $\huG{\fX}$ to $\lu{\fX}$.
\end{enumerate}
\end{proposition}

\begin{proof}
First we prove \rmi.
For each pair $y,y_0$ of adjacent vertices, define
the function $a_{y,y_0}$ by $\de_y-\de_{y_0}$.  Clearly $a_{y,y_0}$
is a multiple of an $\hu{\mu}$-atom, and
$$
\begin{aligned}
\cT a_{y,y_0} (w)
&  = \sum_{z \in \fX} a_{y,y_0}(z) \, k(w, z) \\
&  =  k(w, y) - k(w, y_0)
\quant w \in \fX.
\end{aligned}
$$
The assumption $\cT$ bounded from $\hu{\mu}$ to $\lu{\mu}$
forces $\norm{\cT a_{y,y_0}}{1}$ to be uniformly bounded with
respect to all $y$ and $y_0$ such that $y\sim y_0$.  Thus,
$$
\max_{y_0} \sum_{y\sim y_0} \, \, 
\sum_{w\in \fX} \mod{k(w, y) - k(w, y_0)}
\leq C \, \opnorm{\cT}{H^1;L^1},
$$
as required.

Next we prove \rmii.  Fix a reference point $o$,
and denote by $\eta: \fX \to \BC$ the function defined by 
$$
\eta (x) 
= k(x, o).
$$
Clearly $\eta$ is a radial function, i.e., it depends only
on the distance of $x$ from $o$.  Suppose that $x$ is a point
at distance $j$ from $o$ and that $y\sim o$.  Then the distance
from $x$ to $y$ is either $j-1$ or $j+1$.  Furthermore, 
there are exactly $q$ vertices $y$ adjacent to $o$ 
such that $\rho(x,y) = j+1$ and only one vertex adjacent to $o$ such that 
$\rho(x,y) = j-1$.  Now, by summing in polar co-ordinates centred at $o$,
we see that
\begin{equation} \label{f: polar}
\begin{aligned}
\sum_{y\sim o} \, \,
\sum_{x: \rho(x,o) \geq 2} \mod{k(x,y)-k(x,o)}
& =  \sum_{j=2}^\infty \sum_{x: \rho(x,o) = j} 
     \sum_{y\sim o}\mod{k(x,y)-k(x,o)} \\
& = \sum_{j=2}^\infty \, \,
    \sum_{x: \rho(x,o) = j} \,\,  \sum_{x'\sim x} \mod{\eta (x')-\eta (x)}.
\end{aligned}
\end{equation}
Observe that there exists a constant $C$, depending on $q$, such that
$$
\sum_{x'\sim x} \mod{\eta (x')-\eta (x)}
\geq C \, \Bigl[ \sum_{x'\sim x} \mod{\eta (x')-\eta (x)}^2 \Bigr]^{1/2}.  
$$
and recall that the right hand side is just
$C \, \mod{\nabla\eta(x)}$, by definition of length of the gradient
of $\eta$ (see, for instance, \cite[p.~658]{CG}).
Then, by summing both sides with respect to all $x$ such that 
$\rho(x,o)\geq 2$ and using (\ref{f: polar}), we obtain
$$
\sum_{y\sim o} \, \,
\sum_{x: \rho(x,o) \geq 2} \mod{k(x,y)-k(x,o)}
\geq C\, \sum_{x: \rho(x,o) \geq 2}  \mod{\nabla\eta (x)}.  
$$
Clearly 
$$
\sum_{x: \rho(x,o) < 2}  \mod{\nabla\eta (x)}
< \infty,
$$
because the sum is finite, 
so that we may conclude that $\norm{\,\mod{\nabla\eta}\,}{1}$ 
is finite.   
By the isoperimetric property \cite[Thm~VI.4.2]{Ch1}, 
$\norm{\,\mod{\nabla\eta}\,}{1} \geq C\, \norm{\eta}{1}$,
hence $\eta$ is in $\lu{\fX}$, i.e.,
$$
\sum_{x \in \fX} \mod{k(x,o)}
< \infty,
$$
as required.
This condition clearly implies that $\cT$ is bounded
on $\lu{\fX}$, and the proof of \rmii\ is complete. 

Finally, to prove \rmiii, observe that if
$\cT$ extends to a bounded operator from $\huG{\fX}$
to $\lu{\fX}$, then clearly $\cT$ extends
to a bounded operator from $\hu{\fX}$ to $\lu{\fX}$.

Conversely, suppose that 
$\cT$ extends to a bounded operator from $\hu{\fX}$
to $\lu{\fX}$.  By \rmi\ its kernel $k$
satisfies a H\"ormander integral condition, so that
$k$ satisfies
$$
\sum_{x \in \fX} \mod{k(x,o)}
< \infty
$$
by \rmii.  Therefore
$\cT$ extends to a bounded operator on $\lu{\fX}$,
hence, \emph{a fortiori}, from $\huG{\fX}$ to $\lu{\fX}$.
\end{proof}

\subsection{Riemannian manifolds}

Finally, we consider a connected noncompact Riemannian 
manifold $M$, with spectral gap and Ricci curvature bounded from below.
Recall that a Riemannian manifold $M$
is said to have spectral gap if the bottom of the $L^2$ spectrum
of the associated Laplace--Beltrami operator is strictly positive. 
Such manifolds possess the isoperimetric property (see, for instance,
\cite[Section~9]{CMM1} and the references therein).

Denote by $\mu$ the Riemannian measure of $M$.

\begin{theorem}\label{t: 2}
Suppose that $M$ is as above, that
$\cT$ is a bounded linear operator on $\ld{\mu}$ and that its
kernel $k$ satisfies 
\begin{equation}  \label{f: ***}
C_0
:= \sup_{y \in M} \int_{B(y, 2)^c} \mod{\nabla_1 k(x,y)} \wrt\mu(y)
<\infty.
\end{equation}
Then $\cT$ extends to a bounded operator from $\hu{\mu}$ to $\lu{\mu}$
if and only if $\cT$ extends to a bounded operator from
$\huG{\mu}$ to $\lu{\mu}$.
\end{theorem}

\begin{proof}
Clearly if $\cT$ extends to a bounded operator from $\huG{\mu}$
to $\lu{\mu}$, then it extends to a bounded operator
from $\hu{\mu}$ to $\lu{\mu}$.

Conversely, suppose that $\cT$ extends to a bounded operator
from $\hu{\mu}$ to $\lu{\mu}$.  Then it is uniformly
bounded on $\hu{\mu}$-atoms.  Hence to conclude
the proof of the theorem it suffices to prove that
$\cT$ is uniformly bounded on global atoms. 

It is straightforward to check that
for each ball $B$ of radius $1$, there exists 
a Lipschitz function $\vp_B$ on $M$ such that $\vp_B=1$ on 
$2B$, $\vp_B=0$ on $M\setminus 3B$ and $\norm{\nabla\vp_B}{\infty} \leq 1$
almost everywhere. 

Suppose that $b$ is a global atom supported in a ball $B$ of radius $1$.  
Observe that
\begin{equation}\label{f: v}
\begin{aligned}
\norm{\cT b}{1}
&  \leq \norm{\vp_B\, \cT b}{1}+\norm{(1-\vp_B)\, \cT b}{1}  \\
& \leq  \norm{\ca_{3B}\,\cT b}{1} + \norm{(1-\vp_B)\, \cT b}{1}.
\end{aligned}
\end{equation}
We estimate the two summands above separately.

To estimate the first, we observe that, 
by Schwarz's inequality and the fact that $\mu$ is locally
doubling,
\begin{equation}\label{f: **}
\begin{aligned}
\norm{\ca_{3B}\, \cT b}{1}
& \leq \mu(3B)^{1/2} \, \, \norm{\cT b}{2} \\
& \leq \sqrt{\frac{\mu(3B)}{\mu(B)}} \,  \, \opnorm{\cT}{2} \\ 
& \leq C\, \opnorm{\cT}{2}.
\end{aligned}
\end{equation}

To estimate the second summand we shall use
the analytic Cheeger isoperimetric property \cite{Ch}, 
which states that there exists a constant $C$ such that
\begin{equation} \label{f: isop}
\norm{f}{1}
\leq C \, \norm{\nabla f}{1}
\quant f \in \lu{\mu}.
\end{equation}
Observe that
$$
\nabla \bigl[(1-\vp_B) \, \cT b\bigr]
= -(\nabla\vp_B)\, \cT b + (1-\vp_B) \int_B \nabla_1k(\cdot,y) 
   \, b(y)\wrt\mu(y).
$$
Now we apply (\ref{f: isop}) to $(1-\vp_B)\, \cT b$,
and the triangle inequality in the formula above, and obtain
$$
\norm{(1-\vp_B)\, \cT b}{1}
\leq \norm{\ca_{3B}\, \cT b}{1}
  +  \int_{(2B)^c}\!\!\!
  \wrt\mu(x) \, \Bigmod{\int_B\nabla_1k(x,y)\,  b(y)\wrt\mu(y)}.
$$
We estimate the first summand on the right hand
side as in (\ref{f: **}).  To estimate the
second, we use Tonelli's theorem and obtain that 
$$
\int_{(2B)^c}\!\!\!
  \wrt\mu(x) \, \Bigmod{\int_B\nabla_1k(x,y)\,  b(y)\wrt\mu(y)}
\leq \int_B\wrt\mu(y)\, \mod{b(y)} \, \int_{(2B)^c}
   \, \mod{\nabla_1k(x,y)}\wrt\mu(x).
$$
Thus,
\begin{equation}\label{f: *v}
\norm{(1-\vp_B)\cT b}{1}\leq C\, \opnorm{\cT}{2} + C_0.
\end{equation}
By combining (\ref{f: **}) and (\ref{f: *v}), 
we get that there exists a constant $C$ such that
$$
\norm{\cT b}{1}
\leq C\, \opnorm{\cT}{2} + C_0
$$
for all global atoms, as required to conclude the proof of the theorem.
\end{proof}

Now suppose that $M$ is a connected noncompact \emph{unimodular} Lie group, 
endowed with a left invariant Riemannian metric and denote by $\mu$
the associated Riemannian measure (a constant multiple of the left
Haar measure).  We assume that $M$ has spectral gap.

For each element $X$ in the Lie algebra $\frm$ of $M$, 
denote by $\wt X_\ell$ and $\wt X_r$ the left
invariant and the right invariant vector fields whose
value at $e$ is exactly $X$ respectively.
Write $\check f (z)$ for $f(z^{-1}$.
It is straightforward to check that 
\begin{equation} \label{f: left right inv vector fields}
\wt X_\ell \check f
= -\bigl(\wt X_r f\bigr)\check{\phantom{a}}  
\end{equation}
for all functions $f$ in $C_c^\infty(M)$.
Choose an orthonormal basis $X_1,\ldots,X_n$ of $\frm$
(with respect to the given Riemannian metric).
Then 
\begin{equation}  \label{f: length grad}
\mod{\nabla f}(x)
= \Bigl(\,\sum_{j=1}^n \, \bigmod{\wt {(X_j)}_\ell f(x)}^2 \,\Bigr)^{1/2}.
\end{equation}
where $\mod{\nabla f}(x)$ denotes the length of the Riemannian gradient 
of $f$ at the point $x$.

Suppose that $\cT$ is a left invariant operator,
with kernel $k$; define the \emph{convolution kernel} $K$
of $\cT$ by the rule
$$
K(x) 
= k(x,e)
\quant x \in M,
$$
where $e$ denotes the identity of the group $M$. 
Then
$$
k(x,y)
= K (y^{-1} x)
\quant x,y \in M.
$$
Note that $k$ satisfies the local
H\"ormander condition (\ref{f: hormander I}) if and only if $K$
satisfies the following
$$
\sup_{B} \, \sup_{y\in B}
\int_{(2B)^c} \mod{K(y^{-1}x)-K(x)} \wrt\mu (x)
<\infty,
$$
where $B$ runs over all balls of radius at most $1$ centred at the
identity.
In the case where $k$ is differentiable off the diagonal of $M \times M$,
then $K$ is differentiable off the identity, and 
it is often convenient to express the local H\"ormander condition 
in the following form 
\begin{equation} \label{f: Horm nabla}
\sup_{r \in (0,1]} \,r  \, 
\int_{B(e, r)^c} \Bigl(\,\sum_{j=1}^n \, \bigmod{\wt {(X_j)}_r 
       K(x)}^2 \,\Bigr)^{1/2} \wrt\mu (x)  \\
<\infty.
\end{equation}

\begin{corollary}
Suppose that $M$ is a Lie group as above, that
$\cT$ is a left invariant linear operator, bounded on $\ld{\mu}$,
and that its convolution kernel $K$ satisfies the local 
H\"ormander integral condition (\ref{f: Horm nabla}). 
Then $\cT$ extends to a bounded 
operator from $\huG{\mu}$ to $\lu{\mu}$.
\end{corollary}

\begin{proof}
Observe that if $K$ satisfies (\ref{f: Horm nabla}), 
then the kernel $k$ of $\cT$ satisfies the H\"ormander integral condition 
(\ref{f: hormander I}). 
Hence, by \cite[Thm 8.2]{CMM1}, 
$\cT$ is bounded from $\hu{\mu}$ to $\lu{\mu}$. 

We claim that the kernel $k$ satisfies the inequality
$$
 \sup_{y\in M} \int_{B(y,2)^{c}} \mod{k(x,y)}\wrt\mu(x)
<\infty,
$$
whence the desired conclusion follows by Proposition \ref{p: kernelL1 mms}.\par
 To prove the claim, we observe that by (\ref{f: left right inv vector fields}) and
(\ref{f: length grad})
$$
\begin{aligned}
\int_{B(e, 1)^c} \mod{\nabla \check K(x^{-1})} \wrt\mu (x)
& = \int_{B(e, 1)^c} \Bigl(\,\sum_{j=1}^n \, \bigmod{\wt {(X_j)}_\ell
       \check K(x^{-1})}^2 \,\Bigr)^{1/2} \wrt\mu (x)  \\
& = \int_{B(e, 1)^c} \Bigl(\,\sum_{j=1}^n \, \bigmod{\wt {(X_j)}_r
       K(x)}^2 \,\Bigr)^{1/2} \wrt\mu (x),
\end{aligned}
$$
which is finite because $K$ satisfies the H\"ormander type condition
(\ref{f: Horm nabla}).  Since $M$ is unimodular and $B(e,1)^c$
is invariant under the involution $x\mapsto x^{-1}$,
we may conclude that $\bigmod{\nabla \check K}$ is integrable
on $B(e,1)^c$.
Denote by $\vp$ a smooth cutoff function, which is equal to $1$
in $\OV B(e,1/2)$, and equal to $0$ in $B(e,1)^c$.  Clearly
$\bigmod{\nabla [(1-\vp)\check K]}$ is in $\lu{\mu}$, because
$K$ is differentiable off the origin.  Then the
Cheeger analytic isoperimetric inequality (\ref{f: isop})
implies that $(1-\vp)\check K$ is in $\lu{\mu}$, so that
$$
\int_{B(e,1)^c} \bigmod{\check K} \wrt \mu
<\infty.
$$
By the unimodularity of $M$ we may then conclude that
$$
\int_{B(e,1)^c} \mod{K} \wrt \mu
<\infty.
$$
The claim follows directly from this and the fact that
$k(x,y) = K(y^{-1}x)$.
This concludes the proof.
\end{proof}


\begin{thebibliography}{gMMSTc}

\bibitem[CMM1]{CMM1} A. Carbonaro, G. Mauceri and S. Meda,
$H^1$ and $BMO$ on certain measured metric spaces,
{\tt arXiv:0808.0146v1 [math.FA]}, 
to appear in \emph{Ann. Scuola Norm. Sup. Pisa}. 

\bibitem[CMM2]{CMM2} A. Carbonaro, G. Mauceri and S. Meda,
$H^1$ and $BMO$ for certain locally doubling
metric measure spaces of finite measure, {\tt arXiv:0811.0100v1 [math.FA]},
to appear in \emph{Colloq. Math.} 

\bibitem[Ch]{Ch} I. Chavel, \emph{Riemannian geometry: 
a modern introduction}, Cambridge University Press, 1993.

\bibitem[Ch1]{Ch1} I. Chavel, \emph{Isoperimetric
inequalities.  Differential geometric and analytic
perspectives}, vol. 145 of Cambridge Tract in Mathematics,
Cambridge University Press, 2001.

\bibitem[CG]{CG} T. Coulhon and A. Grigoryan,
Random walks on graphs wilth regular
volume growth, \emph{Geom. Funct. Anal.} \textbf{8} (1998), 656--701.

\bibitem[CW]{CW}
R.R. Coifman and G. Weiss, Extensions of Hardy spaces and their
use in analysis,
\emph{{Bull. Amer. Math. Soc.}} \textbf{83} (1977), 569--645.

\bibitem[DV]{DV} O.~Dragicevic and A. Volberg,
Bellman functions and
dimensionless estimates of Littlewood--Paley type, 
\emph{{J. Oper. Theory}} \textbf{56} (2006), 167--198.


\bibitem[FGS]{FGS} E.B. Fabes, C. Guti\'errez and R. Scotto,
Weak type estimates for the Riesz transforms associated with
the Gaussian measure,
\emph{{Rev. Mat. Iberoamericana}} \textbf{10} (1994), 229--281.

\bibitem[FoS]{FoS} L. Forzani and R. Scotto,
The higher order Riesz transforms for Gaussian measure
need not be weak type (1,1),
\emph{Studia Math.} \textbf{131} (1998), 205--214.

\bibitem[FTN]{FTN}
A. Fig\`a-Talamanca and C. Nebbia,
{\it Harmonic Analysis and Representation Theory for Groups
Acting on Homogeneous Trees}.
London Math. Society Lecture Notes Series, {\bf 162},
Cambridge University Press, Cambridge, U. K., 1991.

\bibitem[GMST1]{GMST1} J. Garcia-Cuerva, G.~Mauceri, P. Sj\"ogren
and J.L.~Torrea,  Higher order Riesz operators for the \OU\ semigroup,
\emph{{Pot. Anal.}} \textbf{10} (1999), 379--407.

\bibitem[GMST2]{GMST2} J. Garcia-Cuerva, G.~Mauceri, P. Sj\"ogren
and J.L.~Torrea,  Spectral multipliers for the \OU\ semigroup,
\emph{{J. D'Analyse Math.}} \textbf{78} (1999), 281--305.

\bibitem[GMMST1]{GMMST1} J. Garc{\'\i}a-Cuerva, G. Mauceri, S. Meda, P.
Sj\"ogren, and J. L.  Torrea,
Functional Calculus for the Ornstein-Uhlenbeck Operator,
\emph{J.  Funct. Anal.}  \textbf{183} (2001), no. 2, 413--450.

\bibitem[GMMST2]{GMMST2}
J.~Garc{\'\i}a-Cuerva, G.~Mauceri, S. Meda, P.~Sj\"ogren, J.L.~Torrea,
Maximal operators for the Ornstein--Uhlenbeck semigroup,
\emph{J. London Math. Soc.} \textbf{67} (2003), 219--234.


\bibitem[G]{G} D. Goldberg,
A local version of real Hardy spaces,
\emph{Duke Math. J.} \textbf{46} (1979), 27--42.



\bibitem[GST]{GST} C.E. Guti\'errez, C. Segovia and J.L. Torrea,
On higher order Riesz transforms for Gaussian measures,
\emph{J. Fourier. Anal. Appl.} \textbf{2} (1996), 583--596.

\bibitem[GU]{GU} C.E. Guti\'errez and  W. Urbina,
Estimates for the maximal operator of the \OU\ semigroup,
\emph{Proc. Amer. Math. Soc.}
\textbf{113} (1991), no. 1, 99--104.

\bibitem[Gun]{Gun} R.F. Gundy,
Sur les transformations de Riesz pour le semigroupe d'\OU,
\emph{C. R. Acad. Sci. Paris Sci. Ser. I Math.} \textbf{303} (1986), 967--970.

\bibitem[Gut]{Gut} C. Guti\'errez,
On the Riesz transforms for Gaussian measures,
\emph{J. Funct. Anal.} \textbf{120} (1994), 107--134.

\bibitem[M1]{M1} P.A.~Meyer, \emph{Transformations de Riesz pour le lois
Gaussiennes},
Springer Lecture Notes in Mathematics \textbf{1059} (1984), 179--193.

\bibitem[MM]{MM}
G.~Mauceri and S. Meda,
$BMO$ and $H^1$ for the Ornstein--Uhlenbeck operator,
\emph{J. Funct. Anal.} \textbf{252} (2007), 278--313.

\bibitem[MMS]{MMS}
G.~Mauceri, S. Meda and P.~Sj\"ogren,
Sharp estimates for the Ornstein--Uhlenbeck operator,
\emph{Ann. Sc. Norm. Sup. Pisa, Classe di Scienze}, Serie IV,
\textbf{} (2004), n. 3, 447--480.

\bibitem[MSV]{MSV} S. Meda, P. Sj\"ogren and M. Vallarino,
On the $H^1$--$L^1$ boundedness of operators,
\emph{Proc. Amer. Math. Soc.} \textbf{136} (2008), 2921--2931.

\bibitem[MPS1]{MPS1} T. Men\'arguez, S. P\'erez and F. Soria,
The Mehler maximal function: a geometric proof
of the weak type 1, \emph{J. London Math. Soc.} (2) 
\textbf{61} (2000), 846--856.

\bibitem[M]{M} P.A.~Meyer, \emph{Note sur le processus d'Ornstein--Uhlenbeck},
Springer Lecture Notes in Mathematics \textbf{920} (1982), 95--132.

\bibitem[Mu]{Mu}  B. Muckenhoupt,
Hermite conjugate expansions, \emph{Trans. Amer. Math. Soc.}
\textbf{139} (1969), 243--260.

\bibitem[Pe]{Pe}  S. P\'erez,  
The local part and the strong type for operators related to the Gauss measure, 
\emph{J. Geom. Anal.} \textbf{11}, no. 3, 491--507.

\bibitem[P]{Pisier} G.~Pisier,  
\emph{Riesz transforms: a simpler analytic proof of
P.A.~Meyer's inequality},
Springer Lecture Notes in Mathematics {\bf 1321}
(1988),  485--501.

\bibitem[PS]{PS} S. P\'erez and F. Soria, 
Operators associated with the Ornstein-Uhlenbeck semigroup,
\emph{J. London Math. Soc.} \textbf{61} (2000), 857--871.

\bibitem[S]{S}
P. Sj\"ogren,  \emph{On the maximal function for the Mehler kernel},
in \emph{Harmonic Analysis, Cortona}, 1982, Springer Lecture Notes
in Mathematics \textbf{992} (1983), 73--82.

\bibitem[St1]{St1} E.M. Stein, \emph{Topics in Harmonic Analysis Related to
the Littlewood--Paley Theory}, Annals of Math. Studies, No. \textbf{63},
Princeton N. J., 1970.

\bibitem[St2]{St2} E.M. Stein, \emph{Harmonic Analysis. Real variable
methods, orthogonality and oscillatory integrals}, Princeton Math. Series
No. \textbf{43}, Princeton N. J., 1993.

\bibitem[T]{T} M.E. Taylor,
Hardy spaces and bmo on manifolds with bounded geometry,
\emph{J. Geom. Anal.} \text{19} (2009), no. 1, 137--190.

\bibitem[U]{U} W. Urbina,
On singular integrals with respect to the Gaussian measure,
\emph{Ann. Sc. Norm. Sup. Pisa, Classe di Scienze, Serie IV},
\textbf{XVIII} (1990), no. 4, 531--567.

\end{thebibliography}
\end{document}